\documentclass[a4paper,10pt]{article}

\usepackage{mathptmx}       
\usepackage{helvet}         
\usepackage{courier}        
\usepackage{type1cm}        
\usepackage{makeidx}         
\usepackage{graphicx}        
\usepackage{multicol}        
\usepackage[bottom]{footmisc}
\usepackage{amsmath, amsfonts, amssymb}
\usepackage{amsthm}
\usepackage{color}
\usepackage{tikz}

\usepackage[utf8]{inputenc}

\newtheorem{thrm}{Theorem}
\newtheorem{prop}{Proposition}

\makeindex      

\newcommand{\ds}{\displaystyle}

\def\R{\mathbb R}
\def\pa{\partial}
\def\Div{\mathrm{div}}

\def\N{\mathbb N}
\def\R{\mathbb R}
\def\T{\mathcal T}
\def\M{\mathcal M}
\def\P{\mathcal P}
\def\m{{\rm m}}
\def\mK{|K|}
\def\msig{|\sigma|}

\def\ts{\tau_{\sigma}}

\def\E{\mathcal E}

\def\EK{\E_{K}}

\def\Eext{\E_{ext}}

\def\somK{\ds\sum_{K\in\T}}

\def\somsig{\ds\sum_{{\sigma \in \E;K=K_{\sigma}}}}
\def\somsige{\ds\sum_{\sigma \in \E}}

\def\NKn{N_K^n}
\def\PKn{P_K^n}
\def\NKnp{N_K^{n+1}}
\def\PKnp{P_K^{n+1}}

\def\NKeq{N_K^{*}}
\def\PKeq{P_K^{*}}

\title{Uniform $L^\infty$ estimates for approximate solutions of the 
bipolar drift-diffusion system}
\author{M.~Bessemoulin-Chatard\footnote{Univ. Nantes, CNRS, UMR 6629 -- Laboratoire Jean Leray, F-44000 Nantes, France, Marianne.Bessemoulin@univ-nantes.fr} \and C.~Chainais-Hillairet\footnote{Univ. Lille, CNRS, UMR 8524 -- Laboratoire Paul Painlev\'e,
F-59000 Lille, France, Claire.Chainais@math.univ-lille1.fr}\and A.~J\"ungel\footnote{Institute for Analysis and Scientific Computing, TU Wien, 
1040 Wien, Austria, juengel@tuwien.ac.at}}

\begin{document}

\maketitle

\begin{abstract} We establish uniform $L^\infty$ bounds for approximate solutions of the drift-diffusion system for electrons and holes in semiconductor devices, computed with the Schar\-fetter-Gummel finite-volume scheme. The proof is based on a Moser iteration technique adapted to the discrete case.
\end{abstract}
\textbf{Keywords: } volume scheme, drift-diffusion, Moser iterations \\[1pt]
{\bf MSC }(2010){\bf:} 65M08, 35B40.

\section{Introduction}

We consider the Van Roosbroeck's bipolar drift-diffusion system on 
$\Omega\times(0,T)$, where $\Omega$ is a domain of $ \R^d$ ($d=2,\,3$):
\begin{subequations}\label{systemDD}
\begin{gather}
\pa_{t}N+\Div(-\nabla N+N\nabla\Psi)=-R(N,P),\label{systemDD_N}\\
\pa_{t}P+\Div(-\nabla P-P\nabla\Psi)=-R(N,P),\label{systemDD_P}\\
-\lambda^{2}\Delta\Psi=P-N+C.\label{systemDD_Psi}
\end{gather}
\end{subequations}
The unknowns are the electron density $N$, the hole density $P$ and  
the electrostatic potential $\Psi$. The doping profile $C(x)$ is given and $\lambda$ is the scaled Debye length. This system is supplemented with initial densities $N_{0}$, $P_{0}$,  Dirichlet boundary conditions on $\Gamma^D$ ($N^D$, $P^D$, $\Psi^D$) and homogeneous Neumann boundary conditions on $\Gamma^N$ (with  $\partial\Omega=\Gamma^D\cup\Gamma^N$, $\Gamma^D\cap\Gamma^N=\emptyset$,  and $\text{m}(\Gamma^D)>0$). The Dirichlet boundary conditions are describing Ohmic contacts, while homogeneous Neumann boundary conditions are for the insulated boundary segments. Dirichlet boundary conditions for \eqref{systemDD} may depend on time, but we
assume time-independent data to simplify.
The recombination-generation rate is written under the following form which includes Shockley--Read--Hall  and  Auger terms:
\begin{equation}\label{forme-gene-R}
R(N,P)=R_{0}(N,P)(NP-1).
\end{equation}
In what follows, we consider the following (standard) assumptions:
\begin{description}
\item[(H1)] $C\in L^\infty(\Omega)$,
\item[(H2)]  $N^D,\,P^D\in L^\infty\cap H^1(\Omega)$, $\Psi^D\in H^1(\Omega)$,
\item[(H3)] $N^D\,P^D=1,$
\item[(H4)] $\exists M>0$ such that $0\leq N_{0},\,P_{0},\,N^D,\,P^D\leq M$ a.e. on $\Omega$,
\item[(H5)] 
$
\exists \bar{R}>0 \text{ such that }0\leq R_{0}(N,P)\leq\bar{R}(1+|N|+|P|)\quad\forall N,\,P\in\R.
$
\end{description}
Hypothesis (H3) means that the boundary data are in thermal equilibrium.
Existence and uniqueness of weak solutions to system \eqref{systemDD} have been proved in \cite{Gajewski1986}. Nonnegativity of the densities and uniform-in-time upper bounds have also been shown in \cite{Gajewski1986}. The proof is based on an approach proposed by Alikakos \cite{Alikakos1979}, closely related to the Moser iteration technique 
\cite{Moser1960}.  
\smallskip

Let $\Delta t>0$ be the time step and let consider an admissible mesh of $\Omega$. It is given by a family $\mathcal{T}$ of control volumes, a family $\mathcal{E}$ of edges (or faces) and a family of points $(x_{K})_{K \in \mathcal{T}}$ which satisfy Definition 9.1 in \cite{Eymard2000}. In the set of edges $ \mathcal{E}$, we distinguish the set of interior edges  $\mathcal{E}_{int}$ from the set of boundary edges $\mathcal{E}_{ext}$. We split $\mathcal{E}_{ext}$ into $ \mathcal{E}_{ext}=\mathcal{E}_{ext}^{D} \cup 
\mathcal{E}_{ext}^{N}$, where $ \mathcal{E}_{ext}^{D}$ and $ \mathcal{E}_{ext}^{N}$
is the set of Dirichlet and Neumann boundary edges, respectively. For a control volume $K \in \mathcal{T}$, we denote by 
$\mathcal{E}_{K}=\mathcal{E}_{int,K}\cup \mathcal{E}_{ext,K}^{D}\cup\mathcal{E}_{ext,K}^{N}$.
For all $\sigma \in \mathcal{E}$, we define $\tau_{\sigma}=\msig/d_{\sigma}$, where $d_{\sigma}=\text{d}(x_{K},x_{L})$ for $\sigma=K|L \in \mathcal{E}_{int}$, and $d_{\sigma}=\text{d}(x_{K},\sigma)$ for $\sigma \in \mathcal{E}_{ext}$.
We also need the following assumptions on the mesh:
\begin{subequations}\label{hyp_mesh}
\begin{gather}
\exists\xi >0 \text{ such that }d(x_{K},\sigma)\geq \xi\,d_{\sigma},\quad\forall K\in\T,\quad \forall\sigma\in\EK, \label{regmesh}\\
\exists c_{0}>0 \text{ such that }\ts\geq c_{0},\quad\forall\sigma\in\E. \label{hyp-ts}
\end{gather}
\end{subequations}
A finite volume discretization for \eqref{systemDD} provides an approximate solution $u_{\T}^{n}=(u_{K}^{n})_{K\in\T}$ for all $n\geq 0$ and approximate boundary values $u_{\E^{D}}=(u_{\sigma})_{\sigma\in\Eext^{D}}$ for $u=N,\,P,\,\Psi$. For any vector $u_{\M}=(u_{\T},u_{\E^{D}})$, we define 
$$
D_{K,\sigma}u=u_{K,\sigma}-u_{K}, \quad  \quad D_{\sigma}u=|D_{K,\sigma}u|, \ \forall K\in\T, \forall \sigma \in\E_K,
$$
where $u_{K,\sigma}$ is either $u_L$ ($\sigma=K|L$), $u_\sigma$ ($\sigma\in\E_{K,ext}^D$) or $u_K$ ($\sigma\in\E_{K,ext}^N$). We also  define the discrete $H^{1}$-seminorm $|\cdot|_{1,\M}$ by
$$
|u_{\M}|_{1,\M}^{2}=\somsige\ts(D_{\sigma}u)^{2},\quad\forall u_{\M}=(u_{\T},u_{\E^{D}}).
$$
We define the initial conditions $N_{K}^0$, $P_{K}^0$ as the mean values of $N_0$ and $P_0$ over $K\in\T$.  The boundary conditions are also approximated by taking the mean values of  $N^D$, $P^D$  and $\Psi^D$ 
over each Dirichlet boundary edge $\sigma\in\E_{ext}^D$.

We are now in the position to define the scheme for \eqref{systemDD}, based on a backward Euler in time discretization. For all $K\in\T$ and $n\geq 0$,
\begin{subequations}\label{scheme}
\begin{align}
&\mK \frac{N_K^{n+1}-N_K^n}{\Delta t}+\ds\sum_{\sigma\in {\mathcal E}_K}{\mathcal F}_{K,\sigma}^{n+1}=-\mK\,R(N_{K}^{n+1},P_{K}^{n+1}),\label{scheme-N}\\
&\mK \frac{P_K^{n+1}-P_K^n}{\Delta t}+\ds\sum_{\sigma\in {\mathcal E}_K}{\mathcal G}_{K,\sigma}^{n+1}=-\mK\,R(N_{K}^{n+1},P_{K}^{n+1}),\label{scheme-P}\\
&-\lambda^2\sum_{\sigma\in {\mathcal E}_K}\tau_{\sigma} D_{K,\sigma}\Psi^{n}=\mK (P_K^{n}-N_K^{n}+C_{K}),\label{scheme-Psi}
\end{align}
\end{subequations}
where ${\mathcal F}_{K,\sigma}^{n+1}$ and ${\mathcal G}_{K,\sigma}^{n+1}$ are the Scharfetter-Gummel fluxes
\begin{subequations}\label{SGfluxes}
\begin{align}
&\label{FLUX-SG-N}
{\mathcal F}_{K,\sigma}^{n+1}=
\tau_{\sigma}\left[ B\left(-D_{K,\sigma}\Psi^{n+1}\right)N_K^{n+1}-B\left(D_{K,\sigma}\Psi^{n+1}\right)N_{K,\sigma}^{n+1} \right],\\
&\label{FLUX-SG-P}
{\mathcal G}_{K,\sigma}^{n+1}=
\tau_{\sigma}\left[ B\left(D_{K,\sigma}\Psi^{n+1}\right)P_K^{n+1}-B\left(-D_{K,\sigma}\Psi^{n+1}\right)P_{K,\sigma}^{n+1} \right],
\end{align}
\end{subequations}
and $B$ is the Bernoulli function 
$B(x)=x/(e^x-1)$ for $x\neq 0$, $B(0)=1$.

 In \cite{Bessemoulin-Chatard2016}, the existence of a solution to scheme \eqref{scheme}-\eqref{SGfluxes} and the boundedness of the approximate densities are shown, but the bounds depend on time and blow up when time goes to infinity. The only case where the result is uniform in time is that of zero doping profile. The purpose of this paper is to adapt the ideas developed in \cite{Gajewski1986,Gajewski1989} to the discrete framework to obtain uniform-in-time $L^\infty$ estimates for the approximate densities obtained with scheme \eqref{scheme}--\eqref{SGfluxes} for general doping profiles. 
Our main result reads as follows.
\begin{thrm}\label{thrm-main}
Let (H1)--(H5) hold and let $\M=(\T,\E,\P)$ be an admissible mesh of $\Omega$ satisfying \eqref{hyp_mesh}. Any solution $(N_{\T}^n,P_{\T}^n,\Psi_{\T}^n)_{n\geq 0}$ to \eqref{scheme}--\eqref{SGfluxes} satisfies 
\vspace*{-0.2cm}
\begin{equation}\label{estLinf}
\exists \kappa >0,\ \forall n\ge 0,\quad \|N_{\T}^n\|_{L^\infty(\Omega)}\leq \kappa \text{ and } \|P_{\T}^n\|_{L^\infty(\Omega)}\leq \kappa. 
\end{equation}
The constant $\kappa$ depends only on the initial and boundary data, $C$, $\lambda$, $\bar{R}$, $\Omega$ and $d$, and on the constants $\xi$ and $c_{0}$ given in \eqref{hyp_mesh}, but not on $n$.
\end{thrm}
This theorem establishes a part of the assumptions needed to prove the exponential decay of approximate solutions given by scheme \eqref{scheme}--\eqref{SGfluxes} towards an approximation of the thermal equilibrium \cite[Theorem 3.1]{Bessemoulin-Chatard2016}. However, a uniform positive lower bound for the densities is also required, which is not easy to prove, and its proof is an open problem.\\
The proof of \eqref{estLinf} applies a Nash-Moser type iteration method based on $L^r$ bounds \cite{Alikakos1979,Moser1960}. Let us mention that this method has already been applied to a discrete setting in \cite{FGL2014}. As we deal here with equations on a bounded domain, we have to take care about the boundary conditions. Therefore, as in \cite{Kowalczyk2005}, we establish \eqref{estLinf} for  $N_M=(N-M)^+$ and $P_M=(P-M)^+$, where $M$ is given in (H4), instead of $N$ and $P$. The proof is detailed in Section \ref{sec-proof}. The uniform-in-time $L^1$ bounds for the densities necessary to initialize the Moser iteration method are obtained thanks to an entropy-entropy production estimate, recalled in Section \ref{sec-preliminary}.

\section{Discrete entropy-entropy production inequality}\label{sec-preliminary}

The thermal equilibrium is a steady state for which the electron and hole current
densities and the recombination-generation term vanish. If  (H3) is satisfied, there exists $\alpha\in\R$ such that the thermal equilibrium is defined by 
\begin{subequations}\label{ET}
\begin{gather}
N^\ast=e^{\alpha+\Psi^\ast},\quad P^\ast=e^{-\alpha-\Psi^\ast},\\
-\lambda^2\Delta\Psi^\ast= e^{-\alpha-\Psi^\ast}-e^{\alpha+\Psi^\ast}+C,\\
\Psi^\ast=\Psi^D \mbox{ on } \Gamma^D,\quad \nabla\Psi^\ast\cdot\nu=0\mbox{ on } \Gamma^N.
\end{gather}
\end{subequations}
An approximation of the thermal equilibrium $(N_\T^\ast, P_\T^\ast,\Psi_\T^\ast)$ is given by
\begin{gather}
-\lambda^2\sum_{\sigma\in {\mathcal E}_K}\tau_{\sigma} D_{K,\sigma}\Psi^\ast=\mK \left(e^{-\alpha-\Psi_{K}^\ast}-e^{\alpha+\Psi_{K}^\ast}+C_{K}\right),\ \forall K\in\T, \label{scheme-Psi-eq}\\
N_{K}^{*}=e^{\alpha+\Psi_{K}^\ast},\quad P_{K}^{*}=e^{-\alpha-\Psi_{K}^\ast},\ \forall K\in\T.\label{scheme-N-P-eq}
\end{gather}
Let $H(x)=x\log x-x+1$. The discrete relative entropy is defined by
\begin{multline}\label{defEdiscret}
\mathbb{E}^{n}=\frac{\lambda^{2}}{2}\left|\Psi_{\M}^{n}-\Psi_{\M}^{*}\right|^{2}_{1,\M}+\somK\mK\Bigl(H(\NKn)-H(\NKeq)-\log\NKeq(\NKn-\NKeq) \\[-0.2cm]
+ H(\PKn)-H(\PKeq)-\log\PKeq(\PKn-\PKeq)\Bigl).
\end{multline}
We also define the discrete entropy production:
\begin{multline}\label{defIdiscret}
\mathbb{I}^{n}=\somsig \ts \left[\min(\NKn,N_{K,\sigma}^{n})\left(D_{\sigma}(\log(N^{n})-\Psi^{n})\right)^{2}\right.\\[-0.2cm]
+\left.\min(\PKn,P_{K,\sigma}^{n})\left(D_{\sigma}(\log(P^{n})+\Psi^{n})\right)^{2}\right]\\
+\sum_{K\in\T}\mK R_{0}(\NKn,\PKn)(\NKn\PKn-1)\log(\NKn\PKn),
\end{multline}
We recall the discrete entropy-entropy production inequality proved in \cite{Chatard2011}.

\begin{prop} For all $n\geq 0$,
\begin{equation}\label{ineq-entropy-discret}
0\leq \mathbb{E}^{n+1}+\Delta t\mathbb{I}^{n+1}\leq \mathbb{E}^n.
\end{equation}
\end{prop}

Summing \eqref{ineq-entropy-discret} over $n$, we have $\mathbb{E}^n \leq \mathbb{E}^0$, which gives a uniform-in-time estimate for $\mathbb{E}^n$. 
Then, if $\M$ satisfies  \eqref{hyp-ts}, and since there exists $C^*$ such that $|\Psi_{\M}^*|_{1,\M}\leq C^*$ (see \cite[Lemma 3.3]{Chainais-Hillairet2007}), we have 
$D_{\sigma}\Psi^{n+1}\leq D$, where $D>0$ only depends on $\mathbb{E}^0$, $\lambda$, $c_{0}$ and $C^*$. The properties of the Bernoulli function ensure that 
\begin{equation}\label{minBDPsi}
\exists \gamma \in (0,1],\quad B(D_{\sigma}\Psi^{n+1})\geq\gamma, \quad \forall \sigma\in\E, \forall n\geq 0.
\end{equation}

\section{Proof of Theorem \ref{thrm-main}}\label{sec-proof}

We set
$$
\begin{gathered}
N_{M,K}^n=(N_K^n-M)^+,\ P_{M,K}^n=(P_K^n-M)^+,\quad \forall K\in\T,\ \forall n\geq 0,\\
\mbox{ and } V_q^n=\sum_{K\in\T} \mK \left[ (N_{M,K}^{n})^{q}+(P_{M,K}^{n})^{q}\right],\quad \forall n\geq 0,\ \forall q\geq 1.
\end{gathered}
$$
We start by establishing the following result about the evolution of $V_{q+1}^n$.
\begin{prop}\label{propineq1dis}
 Let $q\geq 1$. There exist positive constants $\mu$ and $\nu$ only depending on $\|C\|_{\infty}$, $\lambda$, $M$, $\bar{R}$  and $\gamma\in (0,1]$ such that
\begin{multline}\label{ineq1dis}
\frac{1}{\Delta t} \left(V_{q+1}^{n+1}-V_{q+1}^{n}\right)
+\frac{4q}{q+1}\gamma\sum_{\sigma\in\E}\left[(D_{\sigma}(N_M^{n+1})^{\frac{q+1}{2}})^2+(D_{\sigma}(P_M^{n+1})^{\frac{q+1}{2}})^2\right]\\
\leq \mu q V_{q+1}^{n+1} + \nu |\Omega|.
\end{multline}
\end{prop}

\begin{proof}
Multiplying \eqref{scheme-N} (resp. \eqref{scheme-P}) by $(N_{M,K}^{n+1})^q$ (resp. $(P_{M,K}^{n+1})^q$) , summing over $K$ and adding the two equations, we obtain 
$S_1+S_2=S_3$, where $S_1$ contains the discrete time derivatives, $S_2$ the numerical fluxes and $S_3$ the recombination-generation term. Using the elementary
identity $(x-y)x^q\geq (x^{q+1}-y^{q+1})/(q+1)$ for all $x,y\geq 0$ and $q\geq 1$, 
we find that
\begin{equation}\label{minS1}
S_1\geq \frac{1}{q+1}\frac{1}{\Delta t} \left(V_{q+1}^{n+1}-V_{q+1}^{n}\right).
\end{equation}
By a discrete integration by parts on $S_2$, combined with some properties of the Bernoulli function, we have
\begin{multline*}
S_2\geq \frac{4q}{(q+1)^2}\sum_{\sigma\in\E} \ts B(D_\sigma \Psi^{n+1})\left[(D_{\sigma}(N_M^{n+1})^{\frac{q+1}{2}})^2+(D_{\sigma}(P_M^{n+1})^{\frac{q+1}{2}})^2\right]\\
-\frac{q}{q+1}\sum_{\sigma\in\E} \ts D_{K,\sigma}\Psi^{n+1} \left[ D_{K,\sigma}((N_{M}^{n+1})^{q+1})-D_{K,\sigma}((P_{M}^{n+1})^{q+1}))\right]\\
-M\sum_{\sigma\in\E}\ts D_{K,\sigma}\Psi^{n+1}\left[ D_{K,\sigma}((N_{M}^{n+1})^{q})-D_{K,\sigma}((P_{M}^{n+1})^{q}))\right].
\end{multline*}
We perform a discrete integration by parts of the two last sums on the right-hand side, use scheme \eqref{scheme-Psi} and the monotonicity of the functions $x\mapsto ((x-M)^+)^q$ and $x\mapsto ((x-M)^+)^{q+1}$. Combined with \eqref{minBDPsi}, this yields
\begin{multline}\label{minS2}
S_2\geq \frac{4q}{(q+1)^2}\gamma\sum_{\sigma\in\E} \ts \left[(D_{\sigma}(N_M^{n+1})^{\frac{q+1}{2}})^2+(D_{\sigma}(P_M^{n+1})^{\frac{q+1}{2}})^2\right]\\
-\frac{q}{q+1}\frac{\Vert C\Vert_\infty}{\lambda^2}V_{q+1}^{n+1}
-M\frac{\Vert C\Vert_\infty}{\lambda^2}V_{q}^{n+1}.
\end{multline}
Thanks to (H5) and the nonnegativity of the approximate densities, we have 
\begin{multline}\label{ineq2dmembre}
R_{0}(\NKnp,\PKnp)(1-\NKnp\PKnp)\left[\left(N_{M,K}^{n+1}\right)^q+\left(P_{M,K}^{n+1}\right)^q
\right]\\
\leq  \bar{R}(1+2M)\left[(N_{M,K}^{n+1})^q+(P_{M,K}^{n+1})^q\right] 
+\bar{R}\left[
(N_{M,K}^{n+1})^{q+1}+(P_{M,K}^{n+1})^{q+1}\right]\\
+\bar{R}\left[(N_{M,K}^{n+1})^qP_{M,K}^{n+1}+(P_{M,K}^{n+1})^qN_{M,K}^{n+1}\right].
\end{multline}
Then, applying the Young's inequality, we obtain
\begin{align*}
V_q^{n+1} &\leq 
 \frac{1}{q+1}\left(qV_{q+1}^{n+1}+\text{m}(\Omega)\right),\\
\sum_{K\in\T}\mK  
\left[(N_{M,K}^{n+1})^qP_{M,K}^{n+1}+(P_{M,K}^{n+1})^qN_{M,K}^{n+1}\right]
&\leq  V_{q+1}^{n+1}.
\end{align*}
Combining this with \eqref{minS1}, \eqref{minS2} and \eqref{ineq2dmembre} finishes the proof.
\qed
\end{proof}

Now, our aim is to control the term $V_{q+1}^{n+1}$ appearing on the right-hand side of \eqref{ineq1dis}. The discrete Nash inequality \cite[Corollary 4.5]{Bessemoulin-Chatard2015} reads for functions $\chi_\T$ that vanish on a part of the boundary as
$$
\left(\sum_{K\in\T} \mK \chi_K^2\right)^{1+\frac{2}{d}}\leq \frac{{\tilde C}}{\xi}\left(\sum_{\sigma\in\E}\ts (D_\sigma \chi)^2\right)\left(\sum_{K\in\T }\mK \vert \chi_K\vert\right)^{\frac{4}{d}},
$$
where $\xi$ is given in \eqref{regmesh} and ${\tilde C}$ only depends on $\Omega$ and $d$. Thanks to  Young's inequality, it follows for $\varepsilon>0$ that, up to a change of the value of ${\tilde C}$,
$$
\sum_{K\in\T} \mK \chi_K^2\leq \frac{{\tilde C}}{\varepsilon^{d/2}\xi^{d/2}}\left(\sum_{K\in\T }\mK \vert \chi_K\vert\right)^2+
\varepsilon\left(\sum_{\sigma\in\E}\ts (D_\sigma \chi)^2\right).
$$
Applying this inequality to $\chi=\left(N_M^{n+1}\right)^{\frac{q+1}{2}}$ and $\chi=\left(P_M^{n+1}\right)^{\frac{q+1}{2}}$, we have
\begin{equation}\label{ineq3dis}
  V_{q+1}^{n+1}
\leq \frac{{\tilde C}}{(\varepsilon\xi)^{d/2}}
\left(V_{\frac{q+1}{2}}^{n+1}\right)^2
+\varepsilon \sum_{\sigma\in\E} \left[(D_{\sigma}(N_M^{n+1})^{\frac{q+1}{2}})^2+(D_{\sigma}(P_M^{n+1})^{\frac{q+1}{2}})^2\right].
\end{equation}
Arguing similarly as in \cite{DiFrancesco2008} and using the fact that
$\gamma \in (0,1]$, we can find $A>0$ depending only on $\mu$ and hence only on
$\|C\|_{\infty}$, $\lambda$, $M$ and $\bar{R}$ such that 
$$
\frac{\gamma A}{q}\left(\mu q+\frac{\gamma A}{q}\right)\leq \frac{4\gamma q}{q+1},\quad \forall q\geq 1.
$$
Therefore, multiplying \eqref{ineq3dis} by $\mu q+\varepsilon(q)$ with 
$\varepsilon(q)=\gamma A/q$ and adding the resulting equation to \eqref{ineq1dis}, 
we infer that
\begin{equation}\label{ineq4dis}
\frac{V_{q+1}^{n+1}-V_{q+1}^n}{\Delta t} 
\!\leq\! -\varepsilon(q)V_{q+1}^{n+1}\!+\!\nu|\Omega|
\!+\!\frac{{\tilde C}}{\varepsilon(q)^{d/2}\xi^{d/2}}\left(\mu q+\varepsilon(q)\right)
\left( V_{\frac{q+1}{2}}^{n+1} \right)^2.
\end{equation}

Let us now define $W_k^n=V_{2^k}^n$ for all $n\geq 0$ and $k\in\N$.
The definitions of $M$ and the initial condition ensure that $W_k^0=0$ for all 
$k\in\N$. Moreover, the discrete entropy-entropy production inequality \eqref{ineq-entropy-discret} ensures that $\mathbb{E}^n\leq\mathbb{E}^0$ for all $n\geq 0$ and applying the inequalities
$$\forall x,y>0\quad  x\log \frac{x}{y}-x+y\geq (\sqrt{x}-\sqrt{y})^2\geq \frac{x}{2}-y,$$
we deduce a uniform bound of $W_0^n$ for all $n\geq 0$.
With $q=2^k-1=\zeta_k$ and $\varepsilon_k=\gamma A/\zeta_k$, we infer
from \eqref{ineq4dis} that 
\begin{equation}\label{ode}
\frac{W_k^{n+1}-W_k^n}{\Delta t} \leq -\varepsilon_k W_k^{n+1}+B\left(\zeta_k^{d/2}(\zeta_k+\varepsilon_k)(W_{k-1}^{n+1})^2+1\right)
\end{equation}
with 
$
B=\gamma^{-d/2}\max\{ \nu\m(\Omega),\frac{\tilde C}{\xi^{d/2}}A^{-d/2}, \frac{\tilde C}{\xi^{d/2}}A^{-d/2}\mu\}$.
Therefore, if $W_{k-1}^n$ is bounded for all $n$ by $E$, we conclude 
from \eqref{ode} that 
$$
  W_{k}^n\leq \frac{B}{\varepsilon_{k}}\left(\zeta_{k}^{d/2}(\zeta_{k}+\varepsilon_{k})
	E^2+1\right), \quad\forall n\geq 0.
$$
Set $\ds{\delta_{k}=B\zeta_k^{d/2}(\zeta_k+\varepsilon_k)/\varepsilon_k}$. 
As $\zeta_k^{d/2}(\zeta_k+\varepsilon_k)\geq 1$, it follows that
\begin{equation}\label{ineq5}
W_k^n\leq \delta_k (E^2+1), \quad \forall n\geq 0.
\end{equation}
We prove by induction (see \cite{Alikakos1979,Kowalczyk2005}) that for all $k\geq 0$,
$$
 W_k^n\leq 2\delta_k(2\delta_{k-1})^2\cdots (2\delta_1)^{2^{k-1}}{\mathcal K}^{2^k},
\quad \forall n\geq 0,
$$
where ${\mathcal K} = \max (1, \sup_{n\geq 0} W_0^n)$. This is a direct consequence of \eqref{ineq5}, remarking that with $E=2\delta_k(2\delta_{k-1})^2\cdots (2\delta_1)^{2^k-1}{\mathcal K}^{2^k}$ for all $k\ge 0$, we have $1\leq E^2$ (thanks to the definition of ${\mathcal K}$).

To conclude, we first remark that $\delta_k\leq D2^{(2+d/2)k}$ with $D=B/A$. Hence
$$
\prod_{j=0}^{k-1} (2\delta_{k-j})^{2^j}
\leq (2D)^{2^k-1}\cdot 2^{(2+d/2)\sum_{j=0}^{k-1} (k-j)2^j}
\leq (2D)^{2^k}\cdot 2^{(2+d/2)\cdot 2^k\sum_{\ell=1}^{\infty} \ell 2^{-\ell}},
$$
and since $\sum_{\ell=1}^{\infty} \ell 2^{-\ell}=2$, we find that 
$W_k^n\leq (2^{5+d}D{\mathcal K})^{2^k}$.
Taking the power $1/2^k$ of $W_k^n$ we obtain
$$
\Vert N_M^n\Vert_{L^{2^k}(\Omega)}\leq 2^{5+d}D{\mathcal K},\quad 
\Vert P_M^n\Vert_{L^{2^k}(\Omega)}\leq 2^{5+d}D{\mathcal K},\quad \forall n\geq 0,\quad\forall k\in\N,
$$
and passing to the limit $k\to\infty$ gives
$$
\Vert N_M^n\Vert_{L^{\infty}(\Omega)}\leq 2^{5+d}D{\mathcal K},\quad 
\Vert P_M^n\Vert_{L^{\infty}(\Omega)}\leq 2^{5+d}D{\mathcal K},\quad \forall n\geq 0.
$$

\textbf{Acknowledgements:} 
The authors have been partially supported by the bilateral French-Austrian
Amad\'ee-\"OAD project. M. B.-C. thanks the project ANR-14-CE25-0001 Achyl\-les. 
C. C.-H. thanks the team Inria/Rapsodi, the ANR Moonrise and the Labex Cempi (ANR-11-LABX-0007-01) for their support. A.J. acknowledges partial support from   
the Austrian Science Fund (FWF), grants P22108, P24304, and W1245.
\bibliographystyle{plain}
\bibliography{bib-DD-exp}

\end{document}